\documentclass[journal]{IEEEtran}
\usepackage{cite}

\ifCLASSINFOpdf
\else
\fi
\usepackage{amssymb,amsmath,amsthm,graphicx}
\usepackage{enumerate,xcolor}

\newtheorem{theorem}{Theorem}
\newtheorem*{theorem*}{Theorem}
\newtheorem{lemma}[theorem]{Lemma}

\newtheorem{corollary}[theorem]{Corollary}

\theoremstyle{definition}
\newtheorem{example}[theorem]{Example}

\begin{document}
%
\title{Generators of Aggregation Functions and Fuzzy Connectives}
%
%
%

\author{Radom\'ir~Hala\v{s}, 
        Radko~Mesiar,
        and~Jozef~P\'ocs 
\thanks{R. Hala\v{s} is with the Department of Algebra and Geometry, Faculty of Science, Palack\'y University Olomouc, 17. listopadu 12, 771 46 Olomouc, Czech Republic, e-mail: radomir.halas@upol.cz}%
\thanks{R. Mesiar is with the Department of Mathematics and Descriptive Geometry, Faculty of Civil Engineering, Slovak University of Technology in Bratislava, Radlinsk\'eho 11, 810 05 Bratislava 1, Slovakia and University of Ostrava, Institute for Research and Applications of Fuzzy Modeling, NSC Centre of Excellence IT4Innovations, 30. dubna 22, 701 03 Ostrava, Czech Republic, e-mail: radko.mesiar@stuba.sk}
\thanks{J. P\'ocs is with the Department of Algebra and Geometry, Faculty of Science, Palack\'y University Olomouc, 17. listopadu 12, 771 46 Olomouc, Czech Republic and Mathematical Institute, Slovak Academy of Sciences, Gre\v s\'akova 6, 040 01 Ko\v sice, Slovakia, e-mail: pocs@saske.sk}
\thanks{Manuscript received April 19, 2005; revised August 26, 2015.}
\thanks{Preprint of an article published by IEEE Transactions on Fuzzy Systems 24 (2016), 1690-1694. It is available online at: \newline ieeexplore.ieee.org/document/7437423}}

%
%

\markboth{Journal of \LaTeX\ Class Files,~Vol.~14, No.~8, August~2015}%
{Hala\v{s} \MakeLowercase{\textit{et al.}}: Generators of Aggregation Functions and Fuzzy Connectives}
%



\maketitle

\begin{abstract}
We show that the class of all aggregation functions on $[0,1]$ can be generated
as a composition of infinitary sup-operation $\bigvee$ acting on sets with cardinality
not exceeding $\mathfrak{c}$, $b$-medians $\mathsf{Med}_b$, $b\in[0,1[$, and unary aggregation functions $1_{]0,1]}$ and $1_{[a,1]}$, $a\in ]0,1]$.
Moreover, we show that we cannot relax the cardinality of argument sets for suprema to be countable, thus showing a kind of minimality of the introduced generating set. 
As a by product, generating sets for fuzzy connectives, such as fuzzy unions, fuzzy intersections and  fuzzy implications are obtained, too.
\end{abstract}


%
\IEEEpeerreviewmaketitle

\section{Introduction}

Aggregation of finitely many inputs into one representative output value has a long history, though an independent theory of aggregation was established only
recently. For more details, including several historical comments, see \cite{BPC,Calvo_2002,Grabisch_et_al_2009}. Aggregation operators, particularly those representing logical connectives, are widely used in connection with various types of fuzzy sets and their applications, cf. \cite{T2,T3,T4}. In this paper we will deal with aggregation on the unit real interval $[0,1]$ only, though our results can be easily extended to an arbitrary interval $[a,b]\subseteq [-\infty,\infty]$. For a positive integer $n \in \mathbb{N}$, an $n$-ary aggregation function on $[0,1]$ is
a function $f\colon [0,1]^n\to [0,1]$ which is increasing (not necessarily in the strict sense) in each coordinate, and satisfies the boundary conditions 
$f(0,\dots,0) = 0$ and $f(1,\dots,1) = 1$. The symbol $\mathsf{Agg}^n$ denotes the set of all $n$-ary aggregation functions on $[0,1]$, and we put $\mathsf{Agg}=\bigcup_{n\in\mathbb{N}}\mathsf{Agg}^n$. Note that we relax a severe constraint $\mathsf{Agg}^1 = \{\mathrm{id}_{[0,1]}\}$ considered in some sources such as \cite{BPC,Calvo_2002,Grabisch_et_al_2009}.

One of the central problems connected with aggregation functions is how can they be constructed. We can mention several construction methods like transformed aggregation, composed aggregation, weighted aggregation, forming ordinal sums etc., for details we refer to the monograph \cite{Grabisch_et_al_2009}. Each of the above mentioned methods typically relies on a very specific approach and the methods usually have a quite different issue. For example, there is a group of methods characterized by the property "from simple to complex".

To understand better some classes of functions, and having tools for construction of such functions, we often look for a generating set of simple functions, whose compositions allow to construct any function from the considered class. As a typical example from the very basic course of algebra, we can mention that the set $P_n$ of all $n$-ary permutations (recall its cardinality $n!$) has, for any $n>2$, a generating set $G_n = \{(2, 1, 3, \dots , n ),(2, 3, \dots ,n ,1)\}$, consisting of the cyclic permutations, with cardinality $2$. For more details see (\cite{Rotman}, Ex. 2.9, p. 24). Clearly, no single permutation can generate $P_n$ if $n > 2$, since in this case $P_n$ does not form a cyclic group. Not going into details, we recall some examples from the aggregation area:
\begin{itemize}
\item[--] strict triangular norms \cite{Alsina_et_al_2006,Klement_et_al_2000,Schweizer-Sklar_1963} of any dimension are generated by automorphisms $\phi\colon[0,1]\to[0,1]$ and the standard product;

\item[--] generated overlap functions \cite{N1} of any dimension are generated by couples of automorphisms $\phi,\eta\colon[0,1]\to[0,1]$ and the standard product;

\item[--] symmetric bisymmetric cancellative aggregation functions are generated by automorphisms $\phi\colon[0,1]\to[0,1]$ and the standard sum, see \cite{Aczel_1966,Grabisch_et_al_2009};

\item[--] congruence preserving aggregation functions on $[0,1]$, see \cite{HMP} (i.e., Sugeno integrals on finite spaces), are generated by projections,
constant functions and binary operations $\vee$ and $\wedge$ \cite{Sug74,Grabisch_et_al_2009,HMP}.
\end{itemize}

As we have mentioned above, one of the central construction methods of aggregation functions on $[0,1]$ is so-called composed aggregation, which is based on the standard composition of real functions, see \cite{Grabisch_et_al_2009}. 
The main goal of this note is to show the following:
\begin{itemize}
\item[--] composed aggregation represents a uniform construction method for the whole set $\mathsf{Agg}$ of aggregation functions on $[0,1]$ in the sense that any aggregation function can be generated as a composition of infinitary sup-operation $\bigvee$ acting on sets with cardinality not exceeding $\mathfrak{c}$, $b$-medians $\mathsf{Med}_b$, $b\in[0,1[$, and unary aggregation functions $1_{]0,1]}$ and $1_{[a,1]}$, $a\in]0,1]$.
\item[--] we cannot relax the cardinality of argument sets for suprema to be countable, in case that the cardinality of a generating set does not exceed the continuum.
\end{itemize}

Our results have a big impact to the basics of the aggregation theory, stressing the important role of $b$-medians, and bringing a representation of aggregation functions on $[0,1]$ of any arity by means of
members of the later introduced generating set. Our approach considered in the next section can be seen as a disjunctive one, considering the operation $\bigvee$ for any input set with cardinality not exceeding $\mathfrak{c}$. Based on this result, we introduce also generating sets for fuzzy connectives, such as fuzzy unions (disjunctions), fuzzy intersections (conjunctions) and  fuzzy implications. Note that
by duality, it would be possible to develop a conjunctive approach. This issue is shortly discussed in concluding remarks.

Observe also that our approach is rather similar to the view on the fuzzy sets as nested systems of sets (i.e., alpha-cuts representation). Moreover, the considered basic aggregation functions ($b$-medians and $0,1$-valued unary aggregation functions) can be seen as simple cells allowing to construct any of aggregation functions. These functions also allow to construct step-wise approximations of aggregation functions which are easy for the further processing. A similar situation concerns so-called memristors and their applications in computer logic and related domains. Let us notice that the memristor (memory resistor) was a term coined in 1971 by circuit theorist Leon Chua as a missing non-linear passive two-terminal electrical component relating electric charge and magnetic flux linkage, see \cite{Chua}. Nowadays, the ideas of Chua are broadly used as a theoretical background for many practical applications.

\section{Generating set of the class $\mathsf{Agg}$}

The construction method of composed aggregation is based on a function composition. We recall its formal definition. 
Let $A$ be a set and $n\in\mathbb N$ be a positive integer. For any $i< n$, the {\it $i$-th $n$-ary projection} 
is for all $x_1,\dots,x_n\in A$ defined by 
$$p_i^n(x_0,\dots,x_{n-1}):=x_i.$$
Composition forms from one $k$-ary operation $f:A^k\to A$ and $k$ $n$-ary operations $g_1,\dots,g_k:A^n\to A$, 
an $n$-ary operation $f(g_1,\dots,g_k):A^n\to A$ defined by
\begin{equation}\label{eq2}
f\big(g_1,\dots,g_k\big)(\mathbf{x}):=f\big(g_1(\mathbf{x}),\dots,g_k(\mathbf{x})\big),
\end{equation}
for all $\mathbf{x}\in A^n$. 

Let us note that the composition of infinitary functions, i.e., functions with infinitely many arguments, can be defined in a similar way. A set of operations on a set $A$ which contains all the projection operations on $A$ and which is closed under the composition
is called a {\it clone}. For an overview of clone theory we refer to the monograph \cite{Lau}. Moreover, the notion of a clone generalizes that of a monoid 
in a sense that it can be viewed as a set of selfmaps of 
a set $A$ that is closed under composition and containing the identical mapping. Indeed, for $k=n=1$, composition is a usual product of selfmaps.

As intersection of clones is again a clone, for any set $F$ of functions on $A$ we can consider the least clone $[F]$ containing the set $F$. We call 
$F$ a {\it generating set} of a clone if $[F]$ coincides with it.

In order to generate the set (clone) $\mathsf{Agg}$ of aggregation functions, we use the following unary and binary functions: 

For any $a\in [0,1]$ we define $\chi_a\colon [0,1]\to [0,1]$ by
\begin{equation}\label{e1}
\chi_a(x)=
\begin{cases}&1, \text{ if }x\geq a,\; x\neq 0; \\
             &0, \text{ if }x<a \text{ or }x=0.\\   
\end{cases}
\end{equation} 

Obviously, $\chi_a$ is an aggregation function for all $a\in [0,1]$. Moreover, for $a\neq 0$ it represents a characteristic function of the closed interval $[a,1]$, while $\chi_0$ is the characteristic function of the half-open interval $]0,1]$. 
Let us note that in a standard literature (see e.g. \cite{Grabisch_et_al_2009}) for these functions another notation is used, namely  $1_{]0,1]}$ for $\chi_0$ and $1_{[a,1]}$ for $\chi_a$, $a\in ]0,1]$. In order to be consistent with our recent paper \cite{HP}, we shall follow the first one.

Further, for any $b\in [0,1]$ define the $b$-median, see \cite{Fung_Fu_1975,Fodor_1996} and \cite{T1}, $\mathsf{Med}_b\colon [0,1]^2\to [0,1]$ by 
\begin{equation}\label{e2}
\mathsf{Med}_b(x,y)=\mathsf{Med}(x,y,b).
\end{equation}

Note, that it can be easily seen that $\mathsf{Med}_0(x,y)= x\wedge y$ and $\mathsf{Med}_1(x,y)= x\vee y$. Consequently, the binary operations $\wedge$ as well as $\mathsf{Med}_1$ need not be used in the following construction. However, in order to simplify notations we use both operations.

For each $n\in\mathbb{N}$ and $b\in[0,1]$ we use the following functions $G^n_b\colon [0,1]^n\to [0,1]$, defined by induction as follows:

\begin{itemize}
\item $G^1_b(x_0)=\mathsf{Med}_b\big(\chi_0(x_0),\chi_1(x_0)\big)$;
\item $G^2_b(x_0,x_1)=\mathsf{Med}_b\big(\chi_0(x_0\vee x_1),\chi_1(x_0\wedge x_1)\big)$;
\item $G^{n+1}_b(x_0,\dots,x_n)=G^2_b\big(G^n_b(x_0,\dots,x_{n-1}),x_n\big)$ provided $n\geq 2$.
\end{itemize}

Obviously, for each $n\in \mathbb{N}$ the function $G^n_b$ is a composition of binary functions $\vee$, $\wedge$, $\mathsf{Med}_b$, $b\in [0,1]$ and unary functions $\chi_a$, $a\in[0,1]$. In the following lemma we show that $G^n_b$ is the $n$-ary constant aggregation function, i.e., the set of all constant aggregation functions can be generated by the mentioned set of binary and unary aggregation functions. 

\begin{lemma}\label{lemn2}
For any $n\in\mathbb{N}$ and $b\in [0,1]$, $G^n_b$ is an aggregation function and 
$$
G^n_b(\mathbf{x})=
\begin{cases}&0, \text{ if } \mathbf{x}=(0,\dots,0); \\
             &1, \text{ if } \mathbf{x}=(1,\dots,1); \\  
             &b, \text{ otherwise}. 
\end{cases}
$$
\end{lemma}

\begin{proof}
Let $b\in[0,1]$ be an arbitrary element. For $n=1$ it can be easily seen that $G^1_b(0)=\mathsf{Med}_b(0,0)=0$, $G^1_b(1)=\mathsf{Med}_b(1,1)=1$, while if $x\in]0,1[$ then $G^1_b(x)=\mathsf{Med}_b(1,0)=\mathsf{Med}(1,0,b)=b$. Similarly, for $n=2$ we obtain $G^2_b(0,0)=0$ and $G^2_b(1,1)=1$. If $(0,0)\neq\mathbf{x}\neq (1,1)$ then $x_0\vee x_1>0$ and $x_0\wedge x_1<1$. Consequently, $G^2_b(x_0,x_1)=\mathsf{Med}_b(1,0)=b$.

Further, suppose that for $n\geq 2$, the assertion of the lemma is valid. If $\mathbf{x}=(0,\dots,0)\in [0,1]^{n+1}$ is the $(n+1)$-ary vector, then $G^{n+1}_b(\mathbf{x})=G^2_b(0,0)=0$. Similarly for $\mathbf{x}=(1,\dots,1)$ we obtain $G^{n+1}_b(\mathbf{x})=G^2_b(1,1)=1$. Finally, assume that $(0,\dots,0)\neq\mathbf{x}\neq(1,\dots,1)$. If $0<x_n<1$, then $G^{n+1}_b(\mathbf{x})=G^2_b\big(G^n_b(x_0,\dots,x_{n-1}),x_n\big)=b$, since for $G^2_b$ the assertion is valid. If $x_n=0$, then there is an index $i<n$ such that $x_i\neq 0$. According to the induction assumption $G^n_b(x_0,\dots,x_{n-1})$ is equal to $1$ or $b$. Hence, we obtain $G^{n+1}_b(\mathbf{x})=G^2_b(1,0)=b$ or $G^{n+1}_b(\mathbf{x})=G^2_b(b,0)=b$. If $x_n=1$ we obtain $G^{n+1}_b(\mathbf{x})=b$ as well. 
\end{proof}

Denote by $[0,1]^n_{*}$ the set of all elements $\mathbf{a}\in[0,1]^n$ satisfying $(0,\dots,0)\neq\mathbf{a}\neq(1,\dots,1)$.
For any $\mathbf{a}=(a_0,\dots,a_{n-1})\in [0,1]^n$, denote by $J_{\mathbf{a}}=\{0\leq i\leq n-1: a_i\neq 0\}$ the set of indices with non-zero values, i.e., $J_{\mathbf{a}}$ represents the support of the vector $\mathbf{a}$.

Let $n\in\mathbb{N}$ be a positive integer and $f\colon [0,1]^n\to [0,1]$ be an aggregation function. For any $\mathbf{a}\in [0,1]^n_{*}$ we define the function $h_{\mathbf{a}}^f\colon [0,1]^n\to [0,1]$ by putting for $\mathbf{x}=(x_0,\dots,x_{n-1})$

\begin{equation}\label{e3}
h_{\mathbf{a}}^f(\mathbf{x})=G^n_{f(\mathbf{a})}(\mathbf{x})\; \wedge\; \bigwedge_{i\in J_{\mathbf{a}}}\chi_{a_i}(x_i).
\end{equation}

Note that the condition $\mathbf{a}\in [0,1]^n_{*}$ implies $J_{\mathbf{a}}\neq\emptyset$. 
Obviously, the function $h_{\mathbf{a}}^f$ has the same arity as the function $f$. However, from \eqref{e3} it can be easily seen that the function $h_{\mathbf{a}}^f$ can be generated by the previously mentioned set of binary and unary aggregation functions.
The following lemma characterizes the values attained by the particular functions $h_{\mathbf{a}}^f$. 

\begin{lemma}\label{lem1}
Given an aggregation function $f\colon [0,1]^n\to [0,1]$ and $\mathbf{a}\in [0,1]^n_{*}$, $h_{\mathbf{a}}^f$ is an aggregation function such that
for all $\mathbf{x}=(x_0,\dots,x_{n-1})\in [0,1]^n,$

\begin{equation}\label{e4}
h_{\mathbf{a}}^f(\mathbf{x})=\begin{cases}&1, \text{ if }\ \mathbf{x}=(1,\dots,1);\\
                                        &f(\mathbf{a}),\text{ if }\ \mathbf{x}\geq \mathbf{a},\ \mathbf{x}\neq(1,\dots,1);\\   
                                        &0, \text{ if }\ \mathbf{x}\ngeq\mathbf{a}.    
              \end{cases}
\end{equation}
\end{lemma}

\begin{proof}
It can be easily seen that the function $h_{\mathbf{a}}^f$ is nondecreasing in each coordinate and it fulfills the boundary conditions, i.e., $h_{\mathbf{a}}^f(0,\dots,0)=0$ as well as $h_{\mathbf{a}}^f(1,\dots,1)=1$.

Further, assume that $(1,\dots,1)\neq \mathbf{x}\geq\mathbf{a}$. Recall, that $\mathbf{x}\geq\mathbf{a}$ if and only if $x_i\geq a_i$ for all $i\in J_{\mathbf{a}}$. In this case, $\chi_{a_i}(x_i)=1$ for each index $i\in J_{\mathbf{a}}$. Also due to Lemma \ref{lemn2}, the condition $(1,\dots,1)\neq \mathbf{x}\neq (0,\dots,0)$ implies $G^n_{f(\mathbf{a})}(x_0,\dots,x_{n-1})=f(\mathbf{a})$. Hence we obtain 
$$ h_{\mathbf{a}}^f(\mathbf{x})=G^n_{f(\mathbf{a})}(\mathbf{x})\; \wedge\; \bigwedge_{i\in J_{\mathbf{a}}}\chi_{a_i}(x_i)=f(\mathbf{a})\wedge 1=f(\mathbf{a}).$$

If $\mathbf{x}\ngeq\mathbf{a}$, then $x_i\ngeq a_i$ for some index $i\in J_{\mathbf{a}}$. Consequently $\chi_{a_i}(x_i)=0$, which yields 
$$ h_{\mathbf{a}}^f(\mathbf{x})=G^n_{f(\mathbf{a})}(\mathbf{x}) \; \wedge\; \bigwedge_{i\in J_{\mathbf{a}}}\chi_{a_i}(x_i)=G^n_{f(\mathbf{a})}(\mathbf{x})\wedge 0= 0.$$

\end{proof}

\begin{lemma}\label{lem2}
Let $f\colon [0,1]^n \to [0,1]$ be an aggregation function and for all $\mathbf{a}\in [0,1]^n_{*}$, $h_{\mathbf{a}}^f$ be the function defined by formula $(\ref{e3})$. Then 
\begin{equation}\label{eqn1}
f(\mathbf{x})=\bigvee_{\mathbf{a}\in [0,1]^n_{*}}h_{\mathbf{a}}^f(\mathbf{x})
\end{equation}
for all $\mathbf{x}\in [0,1]^n$. 
\end{lemma}  

\begin{proof}
With respect to the previous lemma, if $\mathbf{x}=(0,\dots,0)$ or $\mathbf{x}=(1,\dots,1)$, then $\bigvee_{\mathbf{a}\in [0,1]^n_{*}}h_{\mathbf{a}}^f(\mathbf{x})$ gives the corresponding boundary values. Further, let $\mathbf{x}\in [0,1]^n_{*}$ be an $n$-ary vector. Using (\ref{e4}) of Lemma \ref{lem1} we obtain

\begin{equation*}
\begin{split}
 \bigvee_{\mathbf{a}\in [0,1]^n_{*}} h_{\mathbf{a}}^f(\mathbf{x})&=\bigvee_{\substack{\mathbf{a}\in [0,1]^n_{*}\\ \mathbf{a}\leq \mathbf{x}}} h_{\mathbf{a}}^f(\mathbf{x})\ \vee \ \bigvee_{\substack{\mathbf{a}\in [0,1]^n_{*}\\ \mathbf{a}\nleq \mathbf{x}}} h_{\mathbf{a}}^f(\mathbf{x}) =  \\
& \bigvee_{\substack{\mathbf{a}\in [0,1]^n_{*}\\ \mathbf{a}\leq \mathbf{x}}} f(\mathbf{a}) \ \vee \ \bigvee_{\substack{\mathbf{a}\in [0,1]^n_{*}\\ \mathbf{a}\nleq \mathbf{x}}} 0= \bigvee_{\substack{\mathbf{a}\in [0,1]^n_{*}\\ \mathbf{a}\leq \mathbf{x}}} f(\mathbf{a}).
\end{split}
\end{equation*}

Since the function $f$ is nondecreasing and $\mathbf{x}$ is the greatest element of the set $\{\mathbf{a}\in [0,1]^n_{*}: \mathbf{a}\leq x\}$, it follows that $f(\mathbf{a})\leq f(\mathbf{x})$ for all $\mathbf{a}\in [0,1]^n_{*}$, $\mathbf{a}\leq \mathbf{x}$. Consequently 
$$ \bigvee_{\mathbf{a}\in [0,1]^n_{*}} h_{\mathbf{a}}^f(\mathbf{x})= \bigvee_{\substack{\mathbf{a}\in [0,1]^n_{*}\\ \mathbf{a}\leq \mathbf{x}}} f(\mathbf{a})=f(\mathbf{x}),$$ completing the proof.

\end{proof}

As a consequence of this lemma and according to the remark after the definition of $b$-medians \eqref{e2}, we obtain the following:

\begin{theorem}
The set $\mathsf{Agg}$ of all aggregation functions can be generated by the infinitary operation $\bigvee$, by the functions $\chi_a$, $a\in [0,1]$, defined by \eqref{e1} and by the $b$-medians $\mathsf{Med}_b$, $b\in [0,1[$, defined by \eqref{e2}.
\end{theorem}

\begin{example}
Consider the following ternary aggregation function $f(x_0,x_1,x_2)=x_0\cdot x_1 \cdot x_2$. Expressions \eqref{e3} and \eqref{eqn1} yield the following expression for $f(x_0,x_1,x_2)$
\begin{equation*}
\bigvee_{(a_0,a_1,a_2)\in [0,1]^3_*} \!\!\!\Big(G^3_{a_0\cdot a_1\cdot a_2}(x_0,x_1,x_2)\;\wedge \bigwedge_{i\in J(a_0,a_1,a_2)}\!\!\!\chi_{a_i}(x_i)\Big).
\end{equation*}
Using the fact that $a_0\cdot a_1\cdot a_2=1$ if and only if $a_i=1$ for each $i\in\{0,1,2\}$ and due to infinite distributivity this can be further simplified to
\begin{equation*}
\begin{split}
 f(x_0,x_1,x_2)=\bigvee_{u\in]0,1[}&G^3_u(x_0,x_1,x_2)\; \wedge \\
  \bigvee_{\substack{(a_0,a_1,a_2)\in [0,1]^3_*\\ a_0\cdot a_1\cdot a_2=u}}& \big( \chi_{a_0}(x_0)\wedge \chi_{a_1}(x_1)\wedge \chi_{a_2}(x_2)\big).
\end{split}
\end{equation*} 
Hence, in particular for $u=\frac{1}{2}$ we have the following expression
\begin{equation*}
\begin{split}
 \mathsf{Med}_{\frac{1}{2}}\Big( &\chi_0\big(\mathsf{Med}_{\frac{1}{2}}\big( \chi_0(x_0\vee x_1),\chi_1(x_0\wedge x_1)\big)\vee x_2\big), \\
 &\chi_1(\mathsf{Med}_{\frac{1}{2}}\big( \chi_0(x_0\vee x_1),\chi_1(x_0\wedge x_1)\big)\wedge x_2)\Big)\ \wedge \\
&\bigvee_{\substack{(a_0,a_1,a_2)\in [0,1]^3_*\\ a_0\cdot a_1\cdot a_2=\frac{1}{2}}} \big( \chi_{a_0}(x_0)\wedge \chi_{a_1}(x_1)\wedge \chi_{a_2}(x_2)\big). 
\end{split}
\end{equation*} 
\end{example}

{\bf Remark}. Observe that the representation of an aggregation function $f$ based on \eqref{e4} and \eqref{eqn1} can be seen as a counterpart of cut-representation of fuzzy sets. Besides this fact, when restricting in formula (6) the domain for points from $[0,1]^n_{*}$ to ${I_k}^n_{*}= {I_k}^n \setminus \{(0,\dots,0),(1,\dots,1)\}$, where, for an integer $k$, we have 
$I_k = \{\frac{0}{k},\frac{1}{k},\frac{2}{k},\dots,\frac{k}{k}\}$, we obtain a step-wise lower approximation of the considered aggregation function $f$ which, for sufficiently large $k$, can be used for an effective processing of real data by means of $f$.

As we can see, the above described generating set has cardinality of the continuum $\mathfrak{c}=2^{\aleph_0}$. 
Involving the infinitary operation $\bigwedge$, we can lower the cardinality of the generating set. For a real number $a\in[0,1]$ denote by $S_a=\{q\in\mathbb{Q}:q\leq a\}$ the set of all rationals lower or equal than $a$. Note that due to the density of the set $\mathbb{Q}$ in $\mathbb{R}$, it follows that $\bigvee S_a=a$ for each $a\in [0,1]$. Then $\chi_a(x)=\bigwedge_{q\in S_a} \chi_q(x)$ for all $x\in [0,1]$. Indeed, $\bigwedge_{q\in S_a} \chi_q(x)=1$ if and only if $x\geq q$ for all $q\in S_a$, which is equivalent to $x\geq \bigvee S_a=a$. Moreover, for any $b\in [0,1]$, the $b$-median function $\mathsf{Med}_b$ is nondecreasing and continuous, cf. \cite{Grabisch_et_al_2009}. Hence, given a pair $(x,y)\in [0,1]^2$ we obtain 
\begin{equation*}
\begin{split}
\bigvee_{q\in S_b}\mathsf{Med}_q(x,y)= \bigvee_{q\in S_b}\mathsf{Med}(x,y,q)=\bigvee_{q\in S_b}\mathsf{Med}_x(y,q)=\\
\mathsf{Med}_x\big(y,\bigvee_{q\in S_b}q\big)=
\mathsf{Med}_x(y,b)=\mathsf{Med}_b(x,y).
\end{split}
\end{equation*}

\begin{corollary}
The set $\mathsf{Agg}$ of all aggregation functions can be generated by the countable set consisting of the infinitary operations $\bigvee$ and $\bigwedge$, the functions $\chi_q$, $q\in\mathbb{Q}\cap [0,1]$, defined by $(\ref{e1})$ and the $q$-medians $\mathsf{Med}_q$, $q\in \mathbb{Q}\cap [0,1]$, defined by $(\ref{e2})$. 
\end{corollary}

It is an interesting question, whether $\mathsf{Agg}$ can be generated by a finite set of infinitary aggregation functions.

 As one can see, expression (\ref{eqn1}) involves the operation $\bigvee$ with $\mathfrak{c}$ arguments. Another natural question can be raised: having the generating set of cardinality at most $\mathfrak{c}$, is it possible in general to generate the set $\mathsf{Agg}$ using countable suprema or some operations with countable arguments? In the sequel, we show that for $\mathsf{Agg}^1$ it is the case, while for $\mathsf{Agg}$ not.  

For a unary aggregation function $f$, denote by $$D(f)=\big\{c\in]0,1[: \bigvee_{x<c}f(x)<f(c)\big\}.$$

\begin{lemma}
Let $f\colon [0,1]\to [0,1]$ be an aggregation function. The set $D(f)$ is at most countable.
\end{lemma}

\begin{proof}
Any element $c\in D(f)$ determines a non-empty open interval $]\bigvee_{x<c}f(x),f(c)[$. Evidently, different elements determine pairwise disjoint intervals of this type. Since any such interval contains a rational number, it follows that $D(f)$ is at most countable. 
\end{proof}

\begin{lemma}\label{lemn3}
Let $f\colon [0,1]\to [0,1]$ be a unary aggregation function. Then for all $x\in[0,1]$
$$ f(x)=\bigvee_{q\in\mathbb{Q}\cap]0,1[} h^f_q(x) \vee \bigvee_{c\in D(f)}h^f_c(x).$$
\end{lemma}

\begin{proof}
For $x\in (\mathbb{Q}\cap[0,1])\cup D(f)$ the proof of the assertion is the same as in Lemma \ref{lem2}. Hence, assume that $x$ is irrational with $x\notin D(f)$. Since $\mathbb{Q}$ is a dense subset of $\mathbb{R}$, it follows that 
$$\bigvee_{\substack{q\in \mathbb{Q}\cap]0,1[\\ q<x}} f(q)\;\vee\; \bigvee_{\substack{c\in D(f)\\ c<x}} f(c)=\bigvee_{\substack{c\in ]0,1[\\ c<x}} f(c) = f(x),$$ completing the proof. 
\end{proof}

Recall that $\mathfrak{c}^{\aleph_0}=(2^{\aleph_0})^{\aleph_0}=2^{\aleph_0\cdot \aleph_0}=2^{\aleph_0}=\mathfrak{c}$.

\begin{lemma}\label{lemn1}
$\left|\mathsf{Agg}\right|=2^\mathfrak{c}.$
\end{lemma}

\begin{proof}

Obviously, $\left|[0,1]^n\right|=\mathfrak{c}$ for all $n\in\mathbb{N}$. Hence for each $n\in\mathbb{N}$, $\left|\mathsf{Agg}^n\right|\leq \mathfrak{c}^{\mathfrak{c}}=2^\mathfrak{c}$, which represents the cardinality of all functions from $[0,1]^n$ into $[0,1]$. Consequently, $\left|\mathsf{Agg}\right|= \left| \bigcup_{n=1}^{\infty}\mathsf{Agg}^n \right|\leq \aleph_0 \cdot 2^\mathfrak{c}=2^{\mathfrak{c}}$.

Conversely, let $\varphi\colon [0,1]\to [0,1]$ be an arbitrary function. Define $g_{\varphi}\colon [0,1]^2\to [0,1]$ by 
$$
g_{\varphi}(x,y)=
\begin{cases} 
&0, \mbox{ if }x+y<1; \\
&\varphi(x), \mbox{ if }x+y=1;\\
&1, \mbox{ if }x+y>1. 
\end{cases} 
$$  
Evidently $g_{\varphi}$ is a binary aggregation function for each $\varphi\colon [0,1]\to [0,1]$ and the correspondence $\varphi\mapsto g_{\varphi}$ establishes an injection between the set of the cardinality $2^{\mathfrak{c}}$ and $\mathsf{Agg}$. Consequently, Cantor-Bernstein theorem yields the equality. 
\end{proof}

Let $\omega$ be the first infinite ordinal and $\omega_1$ be the first uncountable ordinal.
We assume the axiom of choice. In this case $\omega_1$ is regular, implying $\sup_{i\in I} \alpha_i<\omega_1$, whenever $I$ is countable and $\alpha_i<\omega_1$ for all $i\in I$, cf. \cite{Jech}.
 
Denote by $\mathsf{Agg}^{\omega}$ the set of all aggregation functions $f\colon [0,1]^{\omega}\to [0,1]$ and we put $\mathsf{Agg}^{\sigma}=\mathsf{Agg}\cup \mathsf{Agg}^{\omega}$. 

Let $0<\alpha,\beta\leq\omega$ be ordinals. For an $\alpha$-ary function $f\colon [0,1]^{\alpha}\to[0,1]$ and a system $(g_i:i<\alpha)$ of $\beta$-ary functions the composition of $f$ on the system $(g_i:i<\alpha)$ of functions will be denoted by $f\circ (g_i:i<\alpha)$. 
If $\alpha=n$ is finite, then 
$$ f\circ (g_i:i<n)=f(g_0,\dots,g_{n-1})$$ determines the classical composition.

In the sequel, we will refer to a set of functions involving at most countably many arguments, which is closed under the composition, as a $\sigma$-clone. 
Given a set $S\subseteq \mathsf{Agg}^{\sigma}$ of aggregation functions, $0<\beta\leq \omega$, the symbol $S/_{\beta}$ will denote all aggregation functions with the domain $[0,1]^{\beta}$. For $S\subseteq \mathsf{Agg}^{\sigma}$ and an $\alpha$-ary aggregation function $f$, $\alpha\leq \omega$, we put 
$$C_f(S)=\{f\circ(g_i:i<\alpha): (\exists \beta\leq\omega)(\forall i<\alpha)g_i\in S/_{\beta}\}.$$ If $F\subseteq\mathsf{Agg}^{\sigma}$ is a set of aggregation functions, we define $C_F(S)=\bigcup_{f\in F}C_f(S)$.

\begin{theorem}\label{thmn1}
Let $F\subseteq \mathsf{Agg}^{\sigma}$ be a set of cardinality at most $\mathfrak{c}$. Then the $\sigma$-clone $C$ generated by the set $F$ has also cardinality at most $\mathfrak{c}$.
\end{theorem}  

\begin{proof}
We proceed by the transfinite recursion up to $\omega_1$ as follows: We put
\begin{itemize}
\item $S_0=\bigcup_{0<\alpha\leq\omega}\{p^{\alpha}_i: i<\alpha\}$, where for each $0<\alpha\leq\omega$, $p^{\alpha}_i\colon [0,1]^{\alpha}\to [0,1]$ denotes the $i$-th projection.
\item $S_{\xi+1}=C_F(S_{\xi})$ for all $\xi<\omega_1$.
\item $S_{\xi}=\bigcup_{\lambda<\xi}S_{\lambda}$ for all limit $\xi<\omega_1$.
\end{itemize}
It can be easily seen that $S_{\lambda}\subseteq S_{\xi}$, provided $\lambda<\xi$. Further, we put $C=\bigcup_{\xi<\omega_1}S_{\xi}$. We show that $C$ forms a $\sigma$-clone generated by the set $F$ and $\left|C\right|=\mathfrak{c}$. Obviously, $C$ contains all projections, while $F\subseteq S_1$. In order to show that $C$ is closed under compositions, assume that $f\in C$ is an $\alpha$-ary function and $(g_i:i<\alpha)\in {C/_{\beta}}$ is a system of $\beta$-ary functions, where $0<\alpha,\beta\leq\omega$. Let $\gamma<\omega_1$ be an ordinal such that $f\in S_{\gamma}$, while for all $i< \alpha$, $\gamma_i$ denotes an ordinal such that $g_i\in S_{\gamma_i}$. Since $\alpha$ is at most countable, it follows that $\delta=\sup_{i<\alpha}\gamma_i<\omega_1$. As $g_i\in S_{\delta}$ for all $i<\alpha$, we obtain $f\circ(g_i:i<\alpha)\in S_{\delta+\gamma}$, showing that $C$ is closed under compositions.

Further, using the transfinite induction, we show that $\left|S_{\xi}\right|\leq \mathfrak{c}$ for all $\xi<\omega_1$. Obviously $\left|S_0\right|=\aleph_0 < \mathfrak{c}$. Assume, that $\left|S_{\lambda}\right|\leq\mathfrak{c}$ for all $\lambda<\xi$. We show $\left|S_{\xi}\right|\leq\mathfrak{c}$ as well. If $\xi=\lambda+1$ for some $\lambda<\omega_1$, then $S_{\xi}=C_F(S_{\lambda})=\bigcup_{f\in F}C_f(S_{\lambda})$. Given an arbitrary $\alpha$-ary function $f\in F$ we obtain $\left|C_f(S_{\lambda})\right|\leq \left|S_{\lambda}\right|^{\left|\alpha\right|}\leq \mathfrak{c}^{\aleph_0}=\mathfrak{c}$. Consequently, $\left|S_{\xi}\right|=\left|C_F(S_{\lambda})\right|\leq\left|F\right|\cdot\mathfrak{c}\leq\mathfrak{c}\cdot\mathfrak{c}=\mathfrak{c}$. If $\xi$ is limit, then $\xi$ is countable and $\left|S_{\xi}\right|\leq \aleph_0\cdot\mathfrak{c}=\mathfrak{c}$.

Finally, since $\aleph_1\leq \mathfrak{c}$, we obtain $\left|C\right|\leq \aleph_1\cdot \mathfrak{c}=\mathfrak{c}$.
\end{proof}

Since $\mathfrak{c}<2^{\mathfrak{c}}=\left|\mathsf{Agg}\right|$, we obtain the following corollary.

\begin{corollary}
The set of all aggregation functions $\mathsf{Agg}$ cannot be generated as a $\sigma$-clone by any set of countable operations with cardinality at most $\mathfrak{c}$.
\end{corollary}

\begin{corollary}
$\left|\mathsf{Agg}^1\right|=\mathfrak{c}$.
\end{corollary}

\begin{proof}
The set $\{\chi_a:a\in[0,1]\}$ represents a subset of $\mathsf{Agg^1}$ of cardinality $\mathfrak{c}$. Moreover, Theorem \ref{thmn1} and Lemma \ref{lemn3} yield $\left|\mathsf{Agg}^1\right|\leq \mathfrak{c}$.
\end{proof}

\section{Conclusion}

We have introduced a generating set of functions generating the class of $\mathsf{Agg}$ of all aggregation functions on $[0,1]$. This generating set consists of suprema
$\bigvee$ acting on input sets with cardinality at most $\mathfrak{c}$, binary aggregation functions $\wedge$ ($\min$) and $\mathsf{Med}_b$, $b \in [0,1]$, and unary aggregation functions
$1_{]0,1]}$ and $1_{[a,1]}$, $a \in ]0,1]$. Observe that an $n$-ary function $f$ is from $\mathsf{Agg}^n$ if and only if its dual $f^d\colon[0,1]^n\to\mathbb{R}$ given by $f^d(x_0,\dots,x_{n-1}) = 1 - f(1-x_0,\dots,1-x_{n-1})$ is from $\mathsf{Agg}^n$, i.e., also the class $\mathsf{Agg}$ is closed under duality (of aggregation functions). This fact allows to introduce a dual generating set of functions generating the class $\mathsf{Agg}$, consisting of infima $\bigwedge$ acting on input sets with cardinality at most $\mathfrak{c}$, binary aggregation functions
$\vee$ ($\max$) and $\mathsf{Med}_b$, $b \in[0,1]$, and unary aggregation functions $1_{\{1\}}$ and
$1_{]a,1]}$, $a \in [0,1[$. Moreover, we have shown the minimality (with respect to the cardinality of input sets for suprema) of the introduced generating set. 
Particular aggregation functions are fuzzy unions (disjunctions) and fuzzy intersections (conjunctions). Obviously, they have the same generating sets as the class $\mathsf{Agg}$. On the other hand, fuzzy implications can be obtained from binary aggregation functions possessing $0$ as the annihilator (zero element) by means of some strong negation on $[0,1]$, e.g. by means of Zadeh's negation $N\colon[0,1]\to [0,1], N(x)=1-x$. Hence, considering the generating set for aggregation functions and the function $N$, we obtain a generating set for fuzzy implications.

We believe that our results will
help to better understanding of the algebraic structure of aggregation functions (fuzzy connectives), as
well as they will be helpful in constructing aggregation functions (fuzzy connectives) with values known
in some fixed points.

\section*{Acknowledgment}

The first author was supported by the international project Austrian Science Fund (FWF)-Grant Agency of the Czech Republic (GA\v{C}R) number I 1923-N25; the second author by the Slovak Research and Development Agency under contract APVV-14-0013 and by the European Regional Development Fund in the IT4Innovations Centre of Excellence project reg. no. CZ.1.05/1.1.00/02.0070; the third author by the ESF Fund CZ.1.07/2.3.00/30.0041, by the Development project of the faculty of Science Palack\'y University Olomouc and by the Slovak VEGA Grant 2/0044/16.
\ifCLASSOPTIONcaptionsoff
  \newpage
\fi

\end{document}